\documentclass[12pt, a4paper, reqno, twoside, makeidx]{amsart}
\usepackage[margin=3cm, marginpar=3cm]{geometry}

\usepackage{verbatim}
\usepackage{amsmath}
\usepackage{amsfonts}
\usepackage{amssymb}
\usepackage{graphicx}
\usepackage{color}
\usepackage{empheq}
\usepackage{braket}
\usepackage[font={small,it}]{caption}
\usepackage{xypic} 
\usepackage[all]{xy}

\usepackage{amssymb}
\newcommand*{\QEDB}{\hfill\ensuremath{\square}}%

\usepackage{tikz}
\usetikzlibrary{calc}

\theoremstyle{remark}

\theoremstyle{plain}
\newtheorem{thm}{Theorem}[section]
\newtheorem{lem}[thm]{Lemma}
\newtheorem{prop}[thm]{Proposition}

\newtheorem{conj}[thm]{Conjecture}
\newtheorem{Question}[thm]{Question}

\theoremstyle{definition}

\newcommand{\Z}{\mathbb{Z}}

\newcommand{\R}{\mathbb{R}}

\newcommand{\Laur}{\mathbb{Z}_2[U,U^{-1}]}

\newcommand{\F}{\mathbb{Z}_2}

\newcommand{\x}{\mathbf{x}}
\newcommand{\y}{\mathbf{y}}

\newcommand{\s}{\mathfrak{s}}

\newcommand{\Spin}{\text{Spin}}
\newcommand{\Spinc}{\text{Spin}^c}

\newcommand{\Char}{\text{Char}(Q_\Gamma)}

\author{Paolo Aceto}
\address{University of Oxford - Mathematical Institute}
\email{Paolo.Aceto@maths.ox.ac.uk}

\author{Antonio Alfieri}
\address{Central European University, Budapest, Hungary}
\email{alfieri\_antonio@phd.ceu.edu}

\begin{document}
\title[On sums of torus knots concordant to alternating knots]{On sums of torus knots concordant \\ to alternating knots}
\begin{abstract}  
We consider the question, asked by  Friedl, Livingston and Zentner, 
of which sums of torus knots are concordant to alternating knots. 
After a brief analysis of the problem in its full generality,
we focus on sums of two torus knots. We describe some effective 
obstructions based on Heegaard Floer homology. 
\end{abstract}

\begin{tiny}
MSC class: 57M25, 57M27 
\end{tiny}

\maketitle
\thispagestyle{empty}
\section{Introduction}
In \cite{FLZalternating} Friedl, Livingston and Zentner studied the knot concordance 
group $\mathcal{C}$ modulo the subgroup $\mathcal{C}_\text{alt} \subset \mathcal{C}$ 
spanned by alternating knots. In particular they ask the following question. 
\begin{Question}
Which sums of torus knots are concordant to alternating knots?
\end{Question}

According to Murasugi \cite{MurasugiAlternatingKnots} the Alexander polynomial of 
an alternating knot is alternating, meaning that its coefficients are non-zero 
and alternate in sign. Using this criterion one can see that a $(p,q)$ torus knot 
is alternating if and only if $(p,q)=(2n+1,2)$ for some $n\geq 0$. 
Murasugi's theorem was re-proved by  Ozsv\' ath and Szab\' o in \cite{OS8} where 
they find that the knot Floer homology $\widehat{HFK}_{i,j}(K)$ of an alternating 
knot $K$ is supported on the diagonal $i-j= \sigma/2$, where $\sigma$ denotes the 
knot signature.  As a consequence of this theorem one can further prove that a linear 
combination of torus knots is alternating if and only if is a sum of $(2n+1,2)$ 
torus knots. Based on these facts we na\" ively formulate the following 
conjecture.

\begin{conj}\label{mainconjecture} A sum of torus knots is concordant to an alternating knot if and only if it is a sum of $(2n+1,2)$ torus knots. 
\end{conj}

Since both the Alexander polynomial, and the rank of the knot Floer homology groups 
are \textit{not} concordance invariants, the arguments based on these basic tools 
cannot be used to prove Conjecture \ref{mainconjecture}. It is therefore necessary to
look for concordance invariants which admit relatively simple descriptions for alternating knots.

In \cite{OSS4} Ozsv\' ath, Stipsicz and Szab\' o associate to a knot $K \subset S^3$ a continuous piecewise linear function $\Upsilon_K : [0, 2] \to \R$ only depending on the concordance type of $K$. Making use of the computation in \cite{OS8} one can prove \cite[Theorem 1.14]{OSS4}  that for an alternating knot $K$ 
\begin{equation}\label{mainobstruction}
\Upsilon_K(t)=  \frac{\sigma(K)}{2} \cdot (1- |t-1|) \ .
\end{equation} 
Based on the restriction imposed by Equation \eqref{mainobstruction}, Friedl, Livingston and Zentner \cite{FLZalternating} proved that torus knots of the form $(n+1,n)$ with $n \geq 3$ are linearly independent in $\mathcal{C}/\mathcal{C}_\text{alt}$. Their method can be adapted to prove the following.

\begin{prop}\label{luckycases}Let $K= m_1 T_{p_1,q_1}\# \dots \# m_k T_{p_k,q_k}$ be a sum of torus knots. Suppose that $p_i >q_i$, and that there is no $q_i$ coefficient appearing with repetitions in the list of coefficients $\{q_1, \dots , q_k \}$. Then $K$ is concordant to an alternating knot if and only if it is a multiple of a $(2n+1,2)$ torus knots.       
\end{prop}

Using the obstruction of Equation \eqref{mainobstruction} one can further see that the following holds.

\begin{prop} \label{positive}
A  sum of positive torus knots is concordant to an alternating knot if and only if it is a sum of $(2n+1,2)$ torus knots. 
\end{prop}

Note that Proposition \ref{luckycases} and \ref{positive} hold more generally for algebraic knots.

It follows from Proposition \ref{positive} that in order to prove 
Conjecture \ref{mainconjecture} one has to deal with connected sums where both positive 
and negative torus knots occur. Going in this direction, as a by-product of his connected sum formula
for knots in involutive knot Floer homology, Zemke \cite{zemke} proved that the 
knot $T_{ 6,5} \# -T_{ 4, 3} \# -T_{ 5,4}$ is not concordant to a thin knot, and 
hence to an alternating knot. 
In \cite{allen}, using the Kim-Livingston secondary upsilon invariant \cite{KimLiv}, 
Allen proved that a knot of the form $K=T_{p+2,p}\# - T_{p+1,p}$, $p \geq 5$ odd, 
is not concordant to a thin knot.
In order to find further evidence for Conjecture \ref{mainconjecture} 
we study sums of two torus knots. Here is our main result.

\begin{thm}\label{mainresult} If a non-alternating sum of two torus knots is concordant to an 
alternating knot then it is of the form $T_{6n+2,3}\#-T_{6n+1,3}$ with $n \geq 1$. 
\end{thm}

We do not know if the knots in the family $T_{6n+2,3}\# -T_{6n+1,3}$ are concordant to alternating knots. However, 
eventhough their topological significance is unclear these knots may serve as useful test cases for the implementation of further obstructions.

We now summarize the crucial steps leading to Theorem \ref{mainresult}. A knot $K \subset S^3$ whose upsilon function 
satisfies Equation \eqref{mainobstruction} is said \textit{upsilon-alternating}. In Section \ref{sectionone} we give a complete characterization of upsilon-alternating sums of two torus knots.

\begin{lem}\label{sums}
 Let $K$ be a non-slice, non-alternating, connected sum of two torus knots. 
 If $K$ is upsilon-alternating then one of the following holds:
\begin{enumerate}
\item $K=T_{2mr+2,r} \# -T_{2mr+1,r}$ with $m\geq 1$,  and $r\geq 5$odd,
\item $K=T_{6c-1,3} \# -T_{6c-2,3}$ with $c\geq 1$,
\item $K=T_{6c+2,3} \# -T_{6c+1,3}$ with $c\geq 1$.
\end{enumerate}
\end{lem}
Recall that, since torus knots are linearly independent in the concordance group, 
the non-sliceness assumption simply means that the connected sum is algebraically 
nontrivial i.e. it is not of the form $T_{p,q}\# -T_{p,q}$.

In Section \ref{sectiontwo} using a connected sum formula for the 
Kim-Livingston secondary invariant \cite{alfieri1} we prove the following theorem,
which deals with the first family in Lemma \ref{sums}.

\begin{thm}\label{uno}
 For $q \geq 1$, and $r\geq 5$ odd, $K=T_{qr+2,r} \# -T_{qr+1,r}$ is not concordant to a Floer thin knot. In particular, a knot of this form is not concordant to an alternating knot.
\end{thm}

The methods of Proposition \ref{uno} cannot be used to prove that the knots in the other families of Lemma \ref{sums} are not concordant to alternating knots. In fact, we have the following proposition.

\begin{prop}\label{failure}For $q \geq 1$ one can find an acyclic, $\Z$-graded, $(\Z \oplus \Z)$-filtered chain complex $A_*$ and a filtered chain homotopy equivalence 
\[CFK^\infty(T_{3q+2,3}) \oplus A_* \simeq CFK^\infty(T_{3q+1,3}) \otimes CFK^\infty(T_{3,2}) \ . \]
The same holds if we substitute $CFK^\infty$ and the tensor product with their involutive counterparts $CFKI^\infty$ \cite{involutive1,zemke}. 
\end{prop}
Consequently, the knots belonging to family (2) and (3) of Lemma \ref{sums} (corresponding respectively to the cases when $q$ is  odd and even) cannot be distinguished from a Floer thin knot by means of any upsilon type invariant \cite{alfieri1}.

Studying the Heegaard Floer correction terms \cite{OS24} of double branched covers
and using a result by Owens and Strle \cite{OwensStrle},  
we prove the following theorem which deals with the second family in Lemma. \ref{sums}. 
\begin{thm}\label{firstfamily}
 Knots of the form $K=T_{6c-1,3} \# -T_{6c-2,3}$ are not concordant to alternating knots.  
\end{thm}

Combining Lemma \ref{sums} with the results of Theorem \ref{uno} and \ref{firstfamily} we obtain a proof of Theorem \ref{mainresult}. 

%We intend to revisit this problem in the future and further investigate the family $\{T_{6n+2,3}\# -T_{6n+1,3}\}$ as well as linear combinations with more than two summands (see Proposition \ref{upsilon} below).
%We include a short discussion about this problem at the end of the paper.

{\small 
\subsection*{Acknowledgements} The authors would like to thanks  Andr\' as I. Stipsicz, Francesco Lin, Brendan Owens, Jennifer Hom, Irving Dai, and Ian Zemke for many useful conversations. Special thanks are also due to Andr\' as Nemethi for generously sharing his expertise.  The first authour was partially supported by ERC grant LTDBud and by MPIM. The second author was partially supported by the NKFIH grant K112735.
\par}

\section{Uspilon-alternating knots}\label{sectionone}
\subsection{Preliminaries on the upsilon invariant}
The upsilon invariant, introduced by Ozsv\' ath, Stipsicz and Szab\' o \cite{OSS4}, 
associates to a knot $K \subset S^3$ a continuous piecewise linear function 
$\Upsilon_K : [0, 2] \to \R$ with the following properties:
\begin{itemize}
\item (Invariance) $\Upsilon_K(t)$ is a knot concordance invariant,
\item(Relation with $\tau$) $\Upsilon_K(t)=- t \cdot \tau(K) $ for $t$ near to zero,
where $\tau$ denotes the concordance invariant introduced by Ozsv\' ath and 
Szab\' o in  \cite{tau}, 
\item (Symmetry) $\Upsilon_K(t)=\Upsilon_K(2-t)$ for all $t \in [0,2]$, 
\item (Additivity) if $K= K_1 \# K_2$ is a connected sum then  
\[\Upsilon_K(t)= \Upsilon_{K_1}(t) + \Upsilon_{K_2}(t) \ , \] 
\item(Mirror) $\Upsilon_{-K}(t)= - \Upsilon_{K}(t)$ where $-K$ denotes the mirror of $K$, 
\item (Slice Genus) if $g_s$ denotes the smooth slice genus then 
\[ \left| \Upsilon_K(t) \right| \leq t g_s(K)   \ .\] 
\end{itemize}
Thus $K \mapsto\Upsilon_K(t)$ descends to a group homomorphism  from the concordance group to $C^0_{PL}[0,2]$
the vector space of continuous piecewise linear functions $[0,2] \to \R$. We will review the definition of $\Upsilon_K(t)$ in Section \ref{sectiontwo} following \cite{Livingston1}. In this section we only need Proposition \ref{orsobalosso} below which provides an algorithm for computing the upsilon function of torus knots.

\begin{prop}[Bodn\' ar \& N\' emethi, Feller \& Krcatovich \cite{feller, Bodnar1}]\label{orsobalosso} Denote by $\Upsilon_{a,b}(t)$ the upsilon function of the torus knot $T_{a,b}$ with $a>b$, then 
\[ \Upsilon_{a,b}(t)= \Upsilon_{a-b,b}(t) +\Upsilon_{a+1,a}(t) \ . \]
Consequently, if  $q_i$ and $r_i$ denote respectively the quotients and the remainders occurring in the Euclidean algorithm for $a$ and $b$ (so that $r_0=a$, $r_{-1}=b$, and $r_{i-1}= q_i r_i + r_{i+1}$), we have that
\[ \Upsilon_{a,b}(t)= \sum_{i=0}^n q_i \cdot \Upsilon_{r_{i}+1, r_i}(t) \ .\]
\end{prop}

The functions $\Upsilon_{i+1,i}(t)$ can be explicitly computed: 
for $t \in [2n/i, 2n+2/i]$ we have that $\Upsilon_{i+1,i}(t)=-n(n+1)-i(i-1-2n)t/2$.
Notice that $\Upsilon_{i+1,i}(t)$ has its first singularity at $t=2/i$. It follows that: 
\begin{itemize}
\item the functions $\left\{ \Upsilon_{i+1,i}(t) \right\}_{i=2}^\infty$ are linearly independent in $C^0_{PL}[0,2]$,
\item if $K$ is a $(p,q)$ torus knot then $\Upsilon_K(t)$ has its first singularity at $t=2/\min(p,q)$.
\end{itemize} 

\begin{proof}[Proof of Proposition \ref{luckycases}]
The upsilon function of a linear combination 
\[K= m_1 T_{p_1,q_1}\# \dots \# m_k T_{p_k,q_k}\]
($p_i>q_i$) has its first singularity at $t=2/q$, where $q= \max_i \big\{q_i \ | \text{ such that } m_i\not=0 \big\}$.
Since the the upsilon function of an alternating knot has at most one singularity at $t=1$, $\Upsilon_K(t)$ has the form of the upsilon function of an alternating knot only if $q=q_1=2$, meaning that $K=m_1  \cdot T_{p_1,2}$.
\end{proof}

\begin{proof}[Proof of Proposition \ref{positive}]
As a consequence of Proposition  \ref{orsobalosso}, the upsilon function of a sum 
of positive torus knots $K$ can be uniquely written as sum
\[ \Upsilon_K(t)= \sum_{i=2}^\infty m_i \cdot \Upsilon_{i+1,i}(t)  \]
with finitely many non-zero $m_i$'s, $m_i \geq 0$, and $m_i=0$ if and only if $\Upsilon_{i+1,i}$ does not appear in any of the expressions of the upsilon functions of the summands of $K$ in terms of the basis $\{\Upsilon_{i+1,i} \}_{i=2}^{\infty}$ of Proposition \ref{orsobalosso}.  

Since the $\Upsilon_{i+1,i}$'s are linearly independent and have their first singularity at $t=2/i$, we have that  $\Upsilon_K(t)$ is a multiple of $f(t)= (1-|1-t|)$ if and only if $m_i=0$ for $i >2$. This forces $K$ to be a sum of $(n,2)$ torus knots proving the claim.
\end{proof}

\subsection{Upsilon-alternating linear combinations of two torus knots} Linear combinations of two torus knots with the upsilon function of an alternating knot can be characterized as follows.
\begin{prop}\label{upsilon} 
Let $K$ be a linear combination of two torus knots. 
Then $K$ is upsilon-alternating if and only if (up to mirroring) one of the 
following holds 
\begin{enumerate}
\item $K$ is slice (since torus knots are linearly independent in the concordance group this can only happen if $K$ is zero as linear combination), 
\item $K$ is alternating, more specifically it is of the form $K=aT_{n,2}\#bT_{m,2}$ for some $a, b \in \Z$ and $m,n >0$ odd,
\item $K=aT_{cbr+1,r}\# -bT_{car+1,r}$ with $a,b,c>0$, $r \geq 3$, and either:
\begin{itemize}
\item $r$ is even
\item $r$ is odd, and $c$ is even,
\item $r$ and $c$ are odd, $a$ and $b$ are even,
\end{itemize}
\item $K=aT_{cbr+2,r}\# -bT_{car+2,r}$ with $a,b,c>0$, $r \geq 3$ odd, and either
\begin{itemize}
\item $c$ is even
\item $c$ is odd, $a$ and $b$ are even,
\item $c$ and $a$ are odd, and $r \equiv 1 \ (\text{mod 4})$,
\end{itemize}
\item $K=aT_{cbr+2, r}\# -bT_{car+1,r}$ with $a,b, c>0$, $r \geq 3$, and either:
\begin{itemize}
\item $r$ is odd, and either $c$ is  even, or $a$ and $b$ are even,
\item $b$ and $c$ are odd, $a$ is even, and $r \equiv 1 \ (\text{mod 4})$,
\item $a,b$ and $c$ are odd, $r \equiv -1 \ (\text{mod 4})$, and $b(r-1)=2a$.  
\end{itemize}
\end{enumerate} 
\end{prop} 
\begin{proof}
Let $K$ be a non-zero linear combination of two torus knots. 
We first characterize those $K$ with $\Upsilon_K(t)$ multiple of $f(t)=1-|1-t|$, then we compute their signature to prove that in fact the relation $\Upsilon(t)= \sigma/2 \cdot (1-|1-t|)$ holds only in the cases displayed in the statement.

Suppose that $\Upsilon_K(t)$ has at most one singularity at $t=1$, meaning that $\Upsilon_K'(t)$ is discontinuous at most at $t=1$. Notice that this is the same as saying that $\Upsilon_K(t)$ is a multiple of $\Upsilon_{3,2}(t)=1-|1-t|$. 

Note that as a consequence of Proposition \ref{positive} and Proposition \ref{luckycases}  
we can assume that $K$ is in the form $aT_{m,n}\# -bT_{m',n}$ for some  positive 
integers $a, b \in \Z$, and a pair of coprime positive integers $(m,n)$ and $(m',n)$. Assume $m>n$ and $m'>n$. 
Denote by $\textbf{r}=(m,n, r_1, \dots, r_s, 1)$, $\textbf{q}=(q_0, \dots, q_s)$ and $\textbf{r}'=(m',n, r_1', \dots, r_s', 1)$, $\textbf{q}'=(q_0', \dots, q_{s'}')$ the vectors of residues and quotients of the Euclidean algorithm for the pairs $(m,n)$ and $(m',n)$ respectively. Set $r_{-1}=m$, $r_{-1}'=m'$, $r_{0}=r_{0}'=n$ and $r_s=r_{s'}=1$ so that 
\[ r_{i-1}= q_i r_i +r_{i+1} \ \ \  r_{j-1}'= q_j' r_j' +r_{j+1}'  \ , \]   
for $i=0, \dots , s$ and $j=0, \dots, s'$.

According to Proposition \ref{orsobalosso} we have
\[\Upsilon_{m,n}(t)= \sum_{i=0}^s q_i \cdot  \Upsilon_{r_i+1, r_i}(t) \ \ \ \ ; \ \ \ 
\Upsilon_{m',n}(t)= \sum_{i=0}^{s'} q_i \cdot  \Upsilon_{r_i'+1, r_i'}(t)  \]
and consequently
\begin{equation}\label{important}
\Upsilon_K(t)= a  \left(\sum_{i=0}^s q_i \cdot\Upsilon_{r_i+1, r_i}(t) \right) - b \left(  \sum_{j=0}^{s'} q_i \cdot \Upsilon_{r_j'+1, r_j'}(t)  \right)\ .
\end{equation}

We now want to solve the functional equation 
$\Upsilon_{m,n}(t)-  \Upsilon_{m',n}(t)= C \cdot \Upsilon_{3,2}(t)$ for $m$ and $m'$.
We distinguish three cases.

\textit{Case I.} $r_s=r_{s'}>2$. By linear independence of the $\Upsilon_{i+1,i}$'s, from Equation \ref{important}  we can conclude that $s=s'$, $r_i=r_i'$ for $i=0, \dots, s$ and $a \mathbf{q}= b \mathbf{q}'$. Assume that $s \geq 1$. By imposing the condition 
\[q_sr_s+1=r_{s-1}=r_{s-1}'=q_s'r_s'+1=q_s'r_s+1 \] 
one can conclude that $q_s=q_s'$, $a=b$ and $m=m'$, \text{i.e.} 
$K=aT_{m,n}-aT_{m,n}$ is slice. If $s=0$ we can conclude that $K$ is in the form 
$aT_{nbc+1,n}\# -bT_{nac+1,n}$.

\textit{Case II.} $r_s=2$ and $r_{s'}>2$. By inspecting Equation \ref{important} 
we see that $s'=s-1$, $r_i=r_i'$ for $i=0, \dots, s-1$ and 
$b  \mathbf{q}'= a \cdot (q_0, \dots, q_{s-1}) $. If $s\geq 2$ then we obtain
\[q_{s-1}r_{s-1}+1=r_{s-2}=r_{s-2}'=q_{s-1}'r_{s-1}'+2=q_{s-1}'r_{s-1}+2\] 
which is a contradiction $(\text{mod } r_s)$. 
Thus $s=1$ and $K$ is in the form $aT_{nbc+2,n}\# -bT_{nac+1,n}$ with $n$ odd.

\textit{Case III.} $r_s=r_{s'}=2$. Because of Equation \ref{important}  we have
$s=s'$, $r_i=r_i'$ for $i=0, \dots, s$ and $a \cdot  (q_0, \dots, q_{s-1})= b \cdot (q_0', \dots, q_{s-1}')$. Assume that $s \geq 2$. By imposing the condition 
\[q_{s-1}r_{s-1}+2=r_{s-2}=r_{s-2}'=q_{s-1}'r_{s-1}'+1=q_{s-1}'r_{s-1}+1\] 
one can conclude that $q_s=q_s'$, $a=b$ and $m=m'$, \text{i.e.} $K=aT_{m,n}\# -aT_{m,n}$ an hence that $K$ slice. Thus either $s=0$ and $K=aT_{m,2}\# -bT_{m',2}$ ($m$ and $m'$ odd), or $s=1$ and  $K=aT_{nbc+2,n}\# -bT_{nac+2,n}$.

Summarising, we have shown that if $K$ is a non-slice linear combination of two torus knots with $\Upsilon_K(t)= C \cdot (1-|1-t|)$ then either:
\begin{itemize}
\item $K=aT_{n,2}\# bT_{m,2}$ with $a, b \in \Z$ and $m,n >0$ odd,
\item $K=aT_{cbr+1,r}\# -bT_{car+1,r}$ with $a,b,c>0$, and $r \geq 3$, 
\item $K=aT_{cbr+2,r}\# -bT_{car+2,r}$ with $a,b,c>0$, and $r \geq 3$ odd, 
\item or $K=aT_{cbr+2, r}\# -bT_{car+1,r}$ with $a,b, c>0$, and $r \geq 3$ odd.
\end{itemize}
 
In order to conclude that the arithmetic conditions in the statement 
are satisfied notice that if $\Upsilon_K(t)= C \cdot (1-|1-t|)$ then $C=-\tau(K)$. 
Consequently,   such a $K$ is upsilon-alternating if and only if the relation $\tau(K)= -\sigma(K)/2$ holds. 

Therefore we need to evaluate the signature of the knots in the list above and compare it with the value of their $\tau$ invariant. The $\tau$ invariant of these knots is particularly easy to compute: $\tau$ is linear, and for the $(p,q)$ torus knot Ozsv\' ath and Szab\' o \cite{tau} proved that $\tau=(p-1)(q-1)/2$. The signature of torus knots on the other hand can be inductively computed as follows \cite{SignaturesTorusKnots}. Let $\sigma(q,r)$ denote the signature of the negative $(q,r)$ torus knot $-T_{q,r}$. 
Extend $\sigma(q,r)$ to the set of pairs $(p,r) \in \Z^2$ with  $r<0$,  $p>|r|$, and $\text{gcd}(p,q)=1$ setting $\sigma(p,r)=\sigma(p,p-r)$. Then 
\begin{equation}
 \sigma(q,r)=(-1)^m \sigma(r, (-1)^m k ) +f(m,r) \ , \label{torusknotssignature}
\end{equation}
where $m$ and $k$ respectively denote the quotient and the residue of the Euclidean division of $q$ and $r$ ($q=m r + k$ with $r>k>0$), and $f(m,r)$ is the function defined by Table \ref{table1}. By means of Equation \ref{torusknotssignature} the computation of the signature of the knots listed above reduces to the one of $\sigma(r,r-1)$ and $\sigma(r,r-2)$. An inductive argument again based on Equation \ref{torusknotssignature} shows that 
$\sigma(r,r-1)= (r-1)^2/2$
when $r\geq3$ is odd, and 
$\sigma(r,r-1)=(r^2-4)/2$
otherwise. Furthermore, 
$ \sigma(r,r-2)=(r-1)^2/2-2$
if $r\equiv -1 \ (\text{mod }4)$, and 
$\sigma(r,r-2)=(r-1)^2/2$
if otherwise $r\equiv 1 \ (\text{mod }4)$. With this said the claimed arithmetic conditions follows immediately.
\end{proof}

\begin{table}[t] 
    \begin{tabular}{ | c | c  | c  |}
    \hline 
    \phantom{$\Big[$} $f(m,r) \ $ & $r$ odd & $r$ even   \\ \hline
    \phantom{$\Big[$} $m$ odd $ \ $ & $\frac{1}{2}\cdot (m+1)(r^2-1) \ $ & $\frac{1}{2} \cdot (mr^2+r^2-4) \ $   \\ \hline
    \phantom{$\Big[$} $m$ even $ \ $ & $\frac{1}{2} \cdot (mr^2-m)$ & $\frac{1}{2} \cdot mr^2$ \\ \hline
    \end{tabular}
    \vspace{0.3cm}
    \caption{\label{table1}}
    \vspace{-0.5cm}
\end{table}

As an immediate corollary of Proposition \ref{upsilon} one gets Lemma \ref{sums}.

\section{Obstructions from the Kim-Livingston secondary invariant} \label{sectiontwo}
In \cite{OS2} Ozsv\' ath and Szab\' o introduced a package of three-manifolds invariants called Heegaard Floer homology. 
This circle of ideas was then used by the same authors \cite{OS7}, and independently by Rasmussen \cite{Ras1} to introduce a knot invariant called knot Floer homology.
For a concise introduction to these topics see \cite{HFKsurvey}.

Recall that knot Floer homology associates to a knot $K \subset S^3$ a finitely-generated, $\Z$-graded, $(\Z \oplus \Z)$-filtered chain complex $CFK^\infty(K)= (\bigoplus_{\x \in B} \F[U, U^{-1}] \cdot \x, \partial )$ with the following properties
\begin{itemize}
\item $\partial$ is $\F[U, U^{-1}]$-linear and given a basis element $\x \in B$,  $\partial \x = \sum_\y n_{\x, \y}U^{m_{\x,\y}} \cdot \y$ for suitable coefficients $ n_{\x, \y} \in \F$, and non-negative exponents $m_{\x, \y} \geq 0$,
\item the multiplication by $U$ drops the homological (Maslov) grading $M$ by two, and the filtration levels (denoted by $A$ and $j$) by one,
\item $H_*(CFK^\infty(K))= \F[U, U^{-1}]$  graded so that $\text{deg}U=-2$.
\end{itemize} 
In \cite{OS7} Ozsv\' ath and Szab\' o show that the filtered chain homotopy type of $CFK^\infty(K)$ only depends on the isotopy class of $K$.
The knot Floer complex $CFK^\infty(K)$ of a knot $K \subset S^3$ can be pictorially described as follows: 
\begin{enumerate}
\item picture each $\F$-generator $U^m \cdot \x$ of $CFK^\infty(K)$ on the planar lattice $\Z \times \Z \subset \R^2$ in position $\left(A(\x)-m, -m \right) \in \Z \times \Z$,
\item label each $\F$-generator $U^m \cdot \x$ of $CFK^\infty(K)$ with its Maslov grading $M(\x)-2m\in \Z$,
\item connect two $\F$-generators $U^n \cdot \x$ and $U^m \cdot \y $ with a directed arrow if in the differential of $U^n \cdot \x$ the coefficient of $U^m \cdot \y$ is non-zero.
\end{enumerate}
The Ozsv\' ath-Stipsicz-Szab\' o upsilon invariant is defined starting form  this picture as follows. 
For $t \in [0,2]$ and $r \in \R$ let $\mathcal{F}_{t,r}$ be the sub-complex of $CFK^\infty(K)$ spanned by the generators contained in the half-plane defined by the equation $t/2 A+ (1-t/2)j\leq r$. Then 
$\Upsilon_K(t)=-2 \cdot \gamma_K(t)$ where $\gamma_K(t)$ is the minimum $r$ for which the inclusion $\mathcal{F}_{t,r} \hookrightarrow CFK^\infty(K)$ induces a surjective map on $H_0$. 

As shown by Kim and Livingston in \cite{KimLiv}, other concordance invariants can be obtained by looking at which filtration levels certain expected homologies are realised. 
This leads to a two variable concordance invariant $\Upsilon^{(2)}_{K,t}(s)$. 
Given $t \in [0,2]$ let $\mathcal{Z}^+$ and $\mathcal{Z}^-$ denote the set of cycles with Maslov grading zero generating $H_0(CFK^\infty(K))$ and contained in $\mathcal{F}_{t+\delta, \gamma_K (t+\delta) }$ and $\mathcal{F}_{t-\delta, \gamma_K (t-\delta) }$ respectively. Since $H_0(CFK^\infty(K)) \simeq \Z_2$ has only one non-zero element, for given  $\xi^+ \in \mathcal{Z}^+$ and $\xi^- \in \mathcal{Z}^-$ there exists a chain with Maslov grading one $\beta \in 
CFK^\infty (K)$ such that $\partial \beta = \xi^+ - \xi^-$. We denote by $\gamma_{K,t}(s)$ the minimum $r$ for which $\mathcal{F}_{s,r}$ contains a $1$-chain realising a homology between a cycle in $\mathcal{Z}^+$ and one in $\mathcal{Z}^-$.
Set $\Upsilon_{K,t}^{(2)}(s)= -2 \cdot (\gamma_{K,t}(s)-\gamma_K(t))$.
Notice that $\Upsilon_{K,t}(s)$ is not defined if $\mathcal{Z}^+ \cap \mathcal{Z}^- \not= \emptyset $, in such a case we set $\Upsilon_{K,t}(s)=-\infty$.

\begin{proof}[Proof of Theorem \ref{uno}] 
This is an argument along the line of \cite[Proposition 1.2]{alfieri1}. More precisely, suppose by contradiction that for some $q \geq 1$, $r \geq 5$ odd, there exists an alternating knot $J$ for which the torus knot  $T_{qr+2,r}$ is concordant to $T_{qr+1,r}\# J$. Then, 
\[ \Upsilon_{T_{qr+2,r}, 4/r}^{(2)}\left(\frac{4}{r}\right)= 
\Upsilon_{T_{qr+1,r}\# J, 4/r}^{(2)}\left(\frac{4}{r}\right)=
\Upsilon_{T_{qr+1,r}, 4/r}^{(2)}\left(\frac{4}{r}\right) \ ,\]
where the first equality is because $\Upsilon^{(2)}$ is a concordance invariant, and the second one is consequence of \cite[Theorem 6.2]{alfieri1}. Notice that \cite[Theorem 6.2]{alfieri1} can be applied at $t=4/r<1$ (here $r>5$ by assumption) since the upsilon function of an alternating knot can have a singularity only at $t=1$. We now show that 
\[\Upsilon_{T_{qr+2,r}, 4/r}^{(2)}\left(\frac{4}{r}\right) \not= 
\Upsilon_{T_{qr+1,r}, 4/r}^{(2)}\left(\frac{4}{r}\right) \ . \]

Given positive integers $a_1, \dots, a_{2k}$ we construct a finitely generated $\Z$-graded, $\Z\oplus \Z$-filtered, chain complex  $C_*(a_1, \dots, a_{2k} )$ as follows. Set 
\[   C_*(a_1, \dots, a_{2k} )= \F\{x_0, \dots , x_k, y_0, \dots , y_{k-1} \} \otimes \Laur  \ ,\] 
and consider the differential  
\[
\begin{cases}
\ \partial x_i= 0 \ \ \ i=0, \dots , k   \\  
\ \partial y_i=x_i+ x_{i+1} \ \ \ i=0, \dots ,  k-1 
 \end{cases} \  .
\]
Define 
\[ \ 
\begin{cases}
 \  A(x_i)=n_i \\ 
 \ j(x_i)= m_i\\
 \ M(x_i)=0 
 \end{cases}
\   \text{ and } \ \ \  \ \
\begin{cases}
\ A(y_i)=n_i \\ 
\ j(y_i)=m_{i+1} \\
\ M(y_i)=1 
\end{cases}  
\] 
where 
\[ \ 
\begin{cases}
 \ n_i=g-\sum_{j=0}^{i}a_{2j} \\ 
 \ n_0=0
 \end{cases}
\  \ \ \  \ \
\begin{cases}
 \ m_i=\sum_{j=1}^i a_{2j-1} \\ 
 \ m_0=0
 \end{cases} \ , 
\] 
and coherently extend these gradings to $\F\{x_0, \dots , x_k, y_0, \dots , y_{k-1} \} \otimes \Laur$ so that the multiplication by $U$ drops the Maslov grading $M$ by two, and the Alexander filtration $A$ as well as the algebraic filtration $j$ by one. The resulting complex is the staircase complex of parameters $a_1, \dots , a_{2g}$ denoted by $C_*(a_1, \dots , a_{2g})$.

The knot Floer complex of a $(p,q)$ torus knot has a representative of the form $ C_*(a_1, \dots, a_{2k} )$. Let $g=(p-1)(q-1)/2$ denotes the four-dimensional genus of the $(p,q)$ torus knot. The semigroup generated by $p$ and $q$, $S_{p,q}=\{np+mq \ | \ n, m \in \Z_{\geq 0} \}$, determines a colouring of $\{0, \dots , 2g-1\}$: color by red the numbers in $S_{p,q} \cap \{0, \dots , 2g-1\} $ and by blue the one in its complement $(\Z \setminus S_{p,q}) \cap  \{0, \dots , 2g-1 \}$. By counting the gaps between blue and red numbers as suggested by Figure \ref{semigroup} we get two sequences of numbers $r_1, \dots r_k$ and $b_1, \dots , b_k$. In \cite{Peters} Petres shows that $CFK^\infty(T_{p,q})\simeq  C_*(r_1,b_1, \dots , r_k, b_k)$.

\begin{figure}
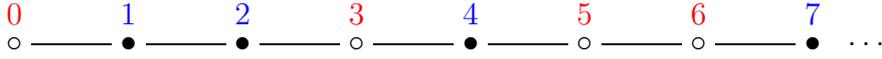

\hspace{0.1cm}
\xygraph{
!{<0cm,0cm>;<1cm,0cm>:<0cm,1cm>::}
!~-{@{-}@[|(2.5)]}
!{(-2,1.5) }*+{\dots}="x0"
!{(-2.7,1.5) }*+{\bullet}="x1"
!{(-4.2,1.5) }*+{\circ}="x2"
!{(-5.7,1.5) }*+{\circ}="x3"
!{(-7.2,1.5) }*+{\bullet}="x4"
!{(-8.7,1.5) }*+{\circ}="x5"
!{(-10.2,1.5) }*+{\bullet}="x6"
!{(-11.7,1.5) }*+{\bullet}="x7"
!{(-13.2,1.5) }*+{\circ}="x8"
!{(-2.7,1.9) }*+{\textcolor{blue}{7}}
!{(-4.2,1.9) }*+{\textcolor{red}{6}}
!{(-5.7,1.9) }*+{\textcolor{red}{5}}
!{(-7.2,1.9) }*+{\textcolor{blue}{4}}
!{(-8.7,1.9) }*+{\textcolor{red}{3}}
!{(-10.2,1.9) }*+{\textcolor{blue}{2}}
!{(-11.7,1.9) }*+{\textcolor{blue}{1}}
!{(-13.2,1.9) }*+{\textcolor{red}{0}}
"x2"-"x1"
"x3"-"x2"
"x4"-"x3"
"x5"-"x4"
"x6"-"x5"
"x7"-"x6"
"x7"-"x8"
}
		\caption{\label{semigroup} The elements of the semigroup generated by $5$ and $3$ correspond to the black dots. The staircase of the torus knot $T_{5,3}$ can be computed from the coloring above by counting the gaps between blue (black dotted) and red (white dotted) numbers. In this case  $r_1=1, r_2=1, r_3=2$, $b_1=2, b_2=1, b_3=1$, and $CFK^\infty(T_{5,3})= S_*(\textcolor{red}{1},\textcolor{blue}{2},\textcolor{red}{1}, \textcolor{blue}{1},\textcolor{red}{2}, \textcolor{blue}{1})$.}
\end{figure}

The semigroup of the torus knot $T_{qr+1,r}$ is given by
\[ S_{qr+1,r}= \bigcup_{j=0}^{r-2} \big\{n \equiv 0\text{ mod }r, \ jpr+1\leq n \leq (r-1)pr \big\} \  \cup \ \Z_{\geq(r-1)pr} \ .\]
Thus $CFK^\infty(T_{qr+1,r})=C_*(1, r-1,\dots 1, r-1, 2 , r-2, \dots , 2, r-2, \dots)$ where the pair $(i, r-i)$ appears $q$ times. Notice that at $t=4/r$ we have  $\mathcal{Z}^+=\{x_q\}$ and $\mathcal{Z}^-=\{x_{2q}\}$. A chain with Maslov grading one realising a homology between $x_q$ and $x_{2q}$ is given by $\beta =\sum_{i=q}^{2q-1} y_i$. Thus 
\begin{align*} \gamma_{T_{qr+1,r},\frac{4}{r}} \left(\frac{4}{r}\right)
&= \min_\xi \left\{\frac{2}{r}\cdot A(\beta + \partial \xi)+\frac{r-2}{r} \cdot j(\beta + \partial \xi) \right\} \\
&= \frac{2}{r}\cdot  A(\beta)+\frac{r-2}{r} \cdot j(\beta)\\ 
&= q(r-2)+2 \left( \frac{r-2}{r} \right) \ ,\end{align*}
where the minimum in the first line is taken over all $\xi \in CFK^\infty(T_{qr+1,r})$ with Maslov grading two. 
Here the second equality is due to the fact that the differential  of $CFK^\infty(T_{qr+1,r})$ vanishes on chains with even Maslov grading. 
The third one is a direct computation. Summing up we obtain 
\begin{align*} \Upsilon^{(2)}_{T_{qr+1,r},\frac{4}{r}} \left(\frac{4}{r}\right)
&= -2 \left( \gamma_{T_{qr+1,r},\frac{4}{r}} \left(\frac{4}{r}\right)- \gamma_{T_{qr+1,r}}\left(\frac{4}{r}\right)\right) \\
&= -2 \left( \gamma_{T_{qr+1,r},\frac{4}{r}} \left(\frac{4}{r}\right) + \frac{1}{2} \Upsilon_{qr+1,r}\left(\frac{4}{r}\right)\right)\\
&=  -2 \left( q(r-2)+2 \left( \frac{r-2}{r} \right) -q(r-2)\right) \ , \end{align*}
which leads to
\[\Upsilon^{(2)}_{T_{qr+1,r},\frac{4}{r}} \left(\frac{4}{r}\right)= -4 \cdot \frac{r-2}{r} \ . \]

For the torus knot $T_{qr+2,r}$ we have $S_{qr+2,r}=S_1 \cup S_2 \cup \Z_{\geq (r-1)(pr+1)}$, where
\[ S_1= \bigcup_{j=0}^{(r-1)/2} \big\{n \equiv 2j \text{ mod }r, \ jpr+1\leq n \leq (r-1)(pr+1) \big\} \]
and
\[ S_2= \bigcup_{j=1}^{(r-1)/2} \Big\{n \equiv 2j-1 \text{ mod }r, \ (r-1-2+j)(pr+1)/2\leq n \leq (r-1)(pr+1) \Big\} \ .\]
Thus $CFK^\infty(T_{qr+2,r})=C_*(1, r-1,\dots 1, r-1, 1, 1, 1, r-3, \dots , 1, 1, 1 , r-3, \dots)$ where the pattern $(1, \dots , 1, r-(2j+1))$ 
with $2j+1$ many $1$'s  appears $q$ times. 
In this case at $t=4/r$ we have  $\mathcal{Z}^+=\{x_q\}$ and $\mathcal{Z}^-=\{x_{3q}\}$. 
A chain with Maslov grading one realising a homology between $x_q$ and $x_{3q}$ is given by $\beta =\sum_{i=q}^{3q-1} y_i$. 
Following the same argument we did for $CFK^\infty(T_{qr+1,r})$, we conclude that 
\[ \gamma_{T_{qr+2,r},\frac{4}{r}} \left(\frac{4}{r}\right)= \frac{2}{r} A(\beta)+\frac{r-2}{r}j(\beta) 
= q(r-2)+ \frac{r-1}{r}+ \left( 2- \frac{6}{r} \right) \]
Thus 
\begin{align*} \Upsilon^{(2)}_{T_{qr+2,r},\frac{4}{r}} \left(\frac{4}{r}\right)
&= -2 \left(  \gamma_{T_{qr+2,r},\frac{4}{r}} \left(\frac{4}{r}\right)- \gamma_{T_{qr+2,r}}\left(\frac{4}{r}\right)\right) \\
&= -2 \left( \gamma_{T_{qr+2,r},\frac{4}{r}} \left(\frac{4}{r}\right) + \frac{1}{2} \Upsilon_{qr+2,r}\left(\frac{4}{r}\right)\right)\\
&=  -2 \left( 2- \frac{6}{r} \right)\ , \end{align*}
which leads to
\[\Upsilon^{(2)}_{T_{qr+2,r},\frac{4}{r}} \left(\frac{4}{r}\right)= -4 \cdot \frac{r-3}{r} \ . \]
Hence, $\Upsilon_{T_{qr+2,r}, 4/r}^{(2)}\left(\frac{4}{r}\right) > 
\Upsilon_{T_{qr+1,r}, 4/r}^{(2)}\left(\frac{4}{r}\right)$ proving the claim.
\end{proof}

The strategy used in the proof of Theorem \ref{uno}  cannot be adapted to deal with the other families of Lemma \ref{sums}.  
To see why this is the case recall the following result.
\begin{thm}[Hom \cite{HFKsurvey}]\label{jen} If two knots $K_1$ and $K_2$ are concordant then there exists $\Z$-graded, $(\Z \oplus \Z)$-filtered, acyclic chain complexes $A_1$ and $A_2$ such that
\begin{equation} \label{HomDependence}
CFK^\infty (K_1) \oplus A_1 \simeq CFK^\infty(K_2) \oplus A_2 \ ,
\end{equation}
where $\simeq$ denotes filtered chain homotopy equivalence. \QEDB
\end{thm}

In \cite{involutive1} Hendricks and Manolescu show that the knot Floer complex $CFK^\infty(K)$ naturally comes with an order four automorphism $\iota_K$ squaring to the Sarkar map \cite{SarkarBasepoint}. In the same paper they prove that the filtered chain homotopy type of the pair $CFKI^\infty(K)=(CFK^\infty(K), \iota_K)$ is an invariant of $K$. In fact, in \cite{involutive2} Hom and Hendricks prove that for $CFKI^\infty$ an analogue of Theorem \ref{jen} holds.   

\begin{thm}[Hendricks \& Hom \cite{involutive2}]\label{involutive2} If two knots $K_1$ and $K_2$ are concordant then there exists $\Z$-graded, $(\Z \oplus \Z)$-filtered, acyclic chain complexes $A_1$ and $A_2$ together with involutions $\iota_{A_1}$ and $\iota_{A_2}$ such that $(CFK^\infty (K_1), \iota_{K_1}) \oplus (A_1,\iota_{A_1}) \simeq (CFK^\infty(K_2), \iota_{K_2}) \oplus (A_2,\iota_{A_2} )$. \QEDB
\end{thm}

We now prove that Equation \eqref{HomDependence} holds in the remaining cases of Lemma \ref{sums}.

\begin{proof}[Proof of Proposition \ref{failure}]With the same notation as in the proof of Theorem \ref{uno} we have that 
$CFK^\infty (T_{3k+1,3})=C_*(1,2, \dots ,1 ,2, 2, 1 , \dots , 2, 1),$ $CFK^\infty (T_{3,2})=C_*(1,1),$
$CFK^\infty (T_{3k+2,3})=C_*(1,2, \dots ,1, 2,1,1, 2, 1 , \dots , 2, 1)$. 
Denote by $z_0, \dots , z_{2q}$ the generators of the staircase chain complex of $CFK^\infty (T_{3k+1,3})$ so that $M(z_{2i})=0$, $M(z_{2i+1})=1$, $\partial z_{2i}=0$, and $\partial z_{2i+1}=z_{2i}+z_{2i+2}$. Similarly, denote by $a, b$ and $c$ the generators of $CFK^\infty(T_{3,2})$ so that $M(a)=M(b)=0$, $M(c)=1$, $\partial a = \partial b =0$, and $\partial c = a +b$. Set 
\[z'_i=
\begin{cases}
 \ a \otimes z_i \ \ \ \ \  \text{ if } i=0, \dots , q\\ 
 \ c \otimes z_q \ \ \ \ \  \text{ if } i=q+1 \\ 
 \  b \otimes z_{i-2} \ \ \  \text{ if } i=q+2, \dots , 2q+2
\end{cases} 
\ \ \ 
\begin{cases}
\alpha_i= c \otimes z_{2i+1}  \\ 
\beta_i= a \otimes z_{2i+1} + c \otimes z_{2i} \\ 
\gamma_i= b \otimes z_{2i+1} + c \otimes z_{2i+2} \\ 
\epsilon_i= b \otimes z_{2i} + a \otimes z_{2i+2} \\
\end{cases} \ .
\]
and notice that  
\begin{itemize}
\item $CFK^\infty(T_{3q+1, 3})= \text{Span}_{\Z_2[U, U^{-1}]}\langle z'_0, \dots , z'_{2q+2} \rangle  \oplus A_*$ where
\[A_*= \bigoplus_{i=1}^q \text{Span}_{\Z_2[U, U^{-1}]}\langle \alpha_i, \beta_i, \gamma_i, \epsilon_i \rangle  \ ,\]
\item $\partial \alpha_i= \beta_i + \gamma_i$, $\partial \beta_i= \partial \gamma_i= \epsilon_i$, $\partial \epsilon_i=0$, and consequently $A_*$ is acyclic (being sum of acyclic complexes), 
\item $M(z_{2i})=0$, $M(z_{2i+1})=1$, $\partial z_{2i}=0$, and $\partial z_{2i+1}=z_{2i}+z_{2i+2}$
which means that $\text{Span}_{\Z_2[U, U^{-1}]}\langle z'_0, \dots , z'_{2q+2} \rangle $ is a staircase complex. 
In fact, a careful check of the Alexander and the algebraic filtrations shows that 
\[CFK^\infty(T_{3q+2,3}) = \text{Span}_{\Z_2[U, U^{-1}]}\langle z'_0, \dots , z'_{2q+2} \rangle \ .\]
\end{itemize}
Summing up we get that $CFK^\infty(T_{3q+1, 3})\otimes CFK^\infty(T_{3,2})= CFK^\infty(T_{3q+2,3}) \oplus A_*$ with $A_*$ acyclic.
Notice that this can also be seen as a consequence of \cite[Lemma 3.18]{StaircaseDependence}. 
To prove the corresponding statement for $CFKI^\infty$ we need to check that
\begin{enumerate}
\item  the involution of $CFKI^\infty(T_{3q+1, 3})\widetilde{\otimes} CFKI^\infty(T_{3,2})$ restricts  to $\iota_{T_{3q+2,3}}$ on the sub complex spanned by $z'_0, \dots , z'_{2q+2}$, 
\item $\iota_{T_{3q+1, 3}} \times \iota_{T_{3, 2}} $leaves $A_*$ invariant. 
\end{enumerate}
Here $CFKI^\infty(T_{3q+1, 3})\widetilde{\otimes} CFKI^\infty(T_{3,2}) = (CFK^\infty(T_{3q+1, 3})\otimes CFK^\infty(T_{3,2}) , \iota_{T_{3q+1, 3}} \times \iota_{T_{3, 2}} )$ denotes the product introduced by Zemke in \cite{zemke}. 
We will adopt Zemke's notation for the rest of this proof. According to \cite[Section 7]{involutive1} the knot involution of a $(p,q)$ torus knot acts on the associated staircase complex as a reflection about the $x=y$ axis. Thus,
\begin{align*}
\iota_{T_{3q+1, 3}} \times \iota_{T_{3, 2}}(z'_i)&=
\iota_{T_{3q+1, 3}} \otimes \iota_{T_{3, 2}} z'_i
+ U^{-1}  (\phi_{T_{3q+1, 3}} \otimes  \psi_{T_{3, 2}})  
\circ (\iota_{T_{3q+1, 3}} \otimes \iota_{T_{3, 2}}) z'_i\\
&= b \otimes z_{2q-i}+  U^{-1}  (\phi_{T_{3q+1, 3}} \otimes  \psi_{T_{3, 2}}) b \otimes z_{2q+2-i} \\
&=z'_{2q+2-i} 
\end{align*}
where the third identity is due to the fact that $\phi_{T_{3, 2}}$ vanishes on $b$. Similarily (using the fact that 
$\phi_{T_{3, 2}}$ vanish on $a$ and $\psi_{T_{3q+1, 3}}$ does so on $z_q$) one proves that 
$\iota_{T_{3q+1, 3}} \times \iota_{T_{3, 2}}(z'_{q+1})= z'_{q+1}$, and $\iota_{T_{3q+1, 3}} \times \iota_{T_{3, 2}}(z'_{i})= z'_{i-2q-2}$ for $i=q+2, \dots , 2q+2$ 
leading to a proof of $(1)$. The proof of $(2)$ is an analogous computation.
\end{proof}

\section{Obstructions from the Owens-Strle theorem} \label{sectionthree}
In this section we deal with the second family of Lemma \ref{sums}. We start by discussing an example in full detail, then we proceed with the necessary computations for the whole family. The main goal of this section is to prove Theorem \ref{firstfamily}.

\subsection{A first example}
It follows from Proposition \ref{upsilon} that the knot $K=T_{5,3} \# -T_{4,3}$ is upsilon-alternating. In order to prove that $T_{5,3} \# -T_{4,3}$ is not concordant to an alternating knot  we will make use of the following result by Owens and Strle \cite{OwensStrle}. 

\begin{thm}[Owens \& Strle]\label{OwSt}
 Let $Y$ be a rational homology sphere with $|H_1(Y;\Z)|=\delta$. If $Y$ bounds a negative-definite four-manifold $X$ and either $\delta$ is square-free or there is no torsion in $H_1(X; \Z)$ then
 \[\max _{\s \in \Spinc(Y)} 4d(Y,\s)\geq \begin{cases}
                                      \ 1-1/\delta \text{ if } \delta \text{ is odd,}\\
                                      \ 1 \text{ if } \delta \text{ is even}
                                     \end{cases}\ . \]
The inequality is strict unless the intersection form of $X$ is $(n-1)\langle-1\rangle\oplus \langle \delta\rangle$. Moreover, the two sides of the inequality are congruent modulo $4/\delta$. \QEDB
\end{thm}

More precisely, we will need the following lemma. 

\begin{lem} \label{criterion} Let $K \subset S^3$ be a knot with $\delta= |\det (K)|$ square-free. Set
\[ d_{\max} (K)= \max _{\s \in \Spinc(\Sigma_2(K))} 4d(\Sigma_2(K),\s) \ , \ \   d_{\min} (K)= \min _{\s \in \Spinc(\Sigma_2(K))} 4d(\Sigma_2(K),\s) \ .\]
If $K$ is concordant to an alternating knot then 
\[ d_{\max} (K) \geq 1- 1/\delta  \ \ \text{ and } \ \ 1/\delta -1 \geq d_{\min}(K) \ .\]
\end{lem}
\begin{proof}
Suppose that $K$ is concordant to an alternating knot $J$. Let $W$ be the double cover of $S^3\times I$ branched along a concordance between $K$ and $J$. 
It is well known that $W$ is a rational homology cobordism between $\Sigma_2(K)$ and $\Sigma_2(J)$. 
By taking the double cover of the four-ball branched along pushed-in copies of the black and the white surface of an alternating diagram of $J$ 
we obtain simply connected definite four-manifolds bounding $\Sigma_2(J)$. By gluing these simply connected definite pieces to $W$ along $\Sigma_2(J)$ 
we obtain a positive-definite filling $X^+$ and a negative-definite filling $X^-$ of $\Sigma_2(K)$. Since $\delta =|\det(K)|=|H_1(\Sigma_2(K);\Z)|$ is a square-free odd number,
we can apply Theorem \ref{OwSt} to the pairs $(X,Y)=(X^-,\Sigma_2(K))$ and $(X,Y)=(-X^+,-\Sigma_2(K))$, and obtain the claimed inequalities.
\end{proof}

\begin{prop}\label{casofacile}
 The knot $T_{5,3}\# -T_{4,3}$ is not concordant to an alternating knot.
\end{prop}
\begin{proof}
Set $K=T_{5,3}\# -T_{4,3}$. Since $3=\det(K)=|H_1(\Sigma_2(K);\Z)|$ is square-free, as a consequence of Lemma \ref{criterion} we have that
 \begin{equation}\label{OwStOb}
d_{\max} (K) \geq 2/3 \ \ \ \text{ and } \ \ \  -2/3 \geq d_{\min}(K) 
 \end{equation}
 We conclude by showing that one of these inequalities does not hold. Notice that $\Sigma_2(K)=\Sigma(2,3,5)\sharp -\Sigma(2,3,4)$ has three $\Spinc$ structures. 
 These are obtained by taking the sum of the spin structure of the Poincar\'e sphere $\Sigma(2,3,5)$ with the three $\Spinc$ structures $\{\s,\mathfrak{t},\overline{\mathfrak{t}}\}$ of $-\Sigma(2,3,4)$. 
The Brieskorn spheres $\Sigma(2,3,5)$ and $\Sigma(2,3,4)$ are respectively the boundaries of the negative-definite $E_8$ and $E_6$ plumbings. 
The $d$-invariants of these graph manifolds where computed in \cite{OSGraphManifolds}, we have
 \[d(\Sigma(2,3,5))=2, \  d(\Sigma(2,3,4),\s)=\frac{3}{2}, \text{ and } \ 
 d(\Sigma(2,3,4),\mathfrak{t})=d(\Sigma(2,3,4),\overline{\mathfrak{t}})=\frac{1}{6} \ . \]
Thus, the $d$-invariants of $\Sigma_2(K)$ are $\{11/6,5/3,5/3\}$ and we can conclude that $d_{\max} (K) = 44/6$ and $ d_{\min}(K) = 20/3$. This  contradicts  the second inequality in \eqref{OwStOb} and proves the claim.
\end{proof}

\subsection{The family $T_{6c-1,3}\# -T_{6c-2,3}$}
In this subsection we prove Proposition \ref{firstfamily}. We do this by generalising the argument used for $T_{5,3}\# -T_{4,3}$ to the knots $K_c=T_{6c-1,3} \# -T_{6c-2,3}$. The double branched cover of each $K_c$ is a difference of two Brieskorn spheres, namely $\Sigma_2(K_c)=\Sigma(2,3,6c-1)\sharp -\Sigma(2,3,6c-2)$. 
Since $|det(K_c)|=3$ the branched double covers $\Sigma_2(K_c)$ have three $\Spinc$ structures (one $\Spin$ and two conjugated $\Spinc$ structures). We will prove that the associated correction terms are $\{11/6,5/3,5/3\}$ independently from $c$. Then the same argument given in the previous section will lead to a contradiction with the inequalities of Lemma \ref{criterion}.

\begin{lem}
For $c\geq 1$ we have $d(\Sigma(2,3,6c-1))=2$. 
\end{lem}
\begin{proof}This is well known. One way to carry out the computation is via the so called knot surgery formula.
Note that $\Sigma(2,3,6c-1)$ can be obtained as $1/c$-surgery along the trefoil knot. As explained in \cite{NiWu},  these correction terms can be computed via the formula $d(S^3_{1/c}(T_{3,2}))=2V_0(T_{3,2})$, where $V_0$ is the first in a sequence of concordance invariants $\{V_i\}_{i\geq 0}$ first introduced by Rasmussen in \cite{Ras1}. For torus knots (and more generally for algebraic knots) these invariants
can be computed combinatorially from the gap function of the semigroup \cite{LivingstonCuspidalCurves}. For the trefoil knot
one has $V_0=1$ and $V_i=0$ for $i\geq 1$.
\end{proof}

In order to compute the correction terms of the Brieskorn spheres $\Sigma(2,3,6c-2)$ we will make use of the algorithm introduced by Ozsv\' ath and Szab\' o in 
\cite{OSGraphManifolds} which we briefly recall. The manifolds 
$\Sigma(2,3,6c-2)$ can be described as the boundary of a negative-definite plumbing of spheres $X_\Gamma$ with associated  star-shaped, three-legged graph $\Gamma$. These particular graphs have at most one bad vertex in the sense of \cite[Definition 1.1]{OSGraphManifolds}. The correction terms of such a plumbed three-manifold $Y_{\Gamma}= \partial X_\Gamma$ can be computed according to the following formula \cite[Corollary 1.5]{OSGraphManifolds}
\begin{equation}\label{correctionterms}
 d(Y_{\Gamma},\s)=\max\frac{K^2+|\Gamma|}{4}\ ,
\end{equation}
where the maximum is taken over all characteristic vectors $K \in H^2(X_\Gamma,\Z)$ representing a $\Spinc$ structure restricting to $\s$. The algorithm given in \cite{OSGraphManifolds} describes how to find a characteristic vector $K$ which maximises the left hand-side of Equation \ref{correctionterms}.

Let $\Gamma$ be a negative-definite plumbing graph with at most one bad vertex. Recall that $K \in H^2(X_\Gamma, \Z)= \text{Hom} (H_2(X_\Gamma, \Z), \Z)$ is characteristic for the intersection pairing $Q_\Gamma$ of $X_\Gamma$ if 
\[\langle K , \alpha \rangle \equiv \alpha^2\text{ mod } 2\]
for every $\alpha \in H_2(X_\Gamma ; \Z)= \bigoplus_{v \in \Gamma} \Z \cdot  v $. 
We denote by $\Char$ the set of characteristic vectors of $Q_{\Gamma}$. We say that $K_0 \in \Char$ is \emph{admissible} if 
\[ m(v)+2 \leq \langle K_0, v \rangle \leq -m(v)\]
for every $v \in \Gamma$, where $m(v)$ denotes the weight of the vertex $v$. Given an admissible vector $K_0$ one can inductively construct a sequence $(K_0, K_1, \dots , K_n)$ in which a term $K_i$ is obtained from its predecessor $K_{i-1}$ by 
summing twice the Poincar\` e dual $PD(v)$ of a vertex $v$ of $\Gamma$ such that $\langle K , v \rangle = -m(v)$.
We will refer to the operation $K \mapsto K+ 2PD(v)$ as a \emph{flip move} at the vertex $v$. A sequence  $(K_0,  \dots , K_n)$ is said to terminate in a \emph{full-path} if one of the following holds
\begin{enumerate}
 \item $m(v) \leq \langle K_n, v \rangle \leq -m(v)-2$ for every vertex $v$,
 \item $ \langle K_n, v \rangle > -m(v)$ for some vertex $v$. 
\end{enumerate}
If the former holds we say that the full-path is \emph{good}, otherwise we say that it is \emph{bad}. 
It follows from \cite[Proposition 3.2 ]{OSGraphManifolds} that the maximum in Equation \ref{correctionterms} is achieved by an admissible characteristic vector representing the $\Spinc$ structure $\s$ and initiating a good full-path. 
In the following lemma  we collect some useful remarks that will be extensively used in the proof of Lemma \ref{famiglia1}. These remarks already appeared in the literature, see for example \cite{SomeBreiskorn}.
\begin{lem}\label{remarks} Let $\Gamma$ be a negative-definite plumbing graph.
\begin{enumerate}
\item If $K_0$ initiates a bad full-path then any other full-path starting with $K_0$ is bad.
\item If $\Gamma' \subset \Gamma$ is a connected subgraph and $K_0' \in \text{Char}(X_{\Gamma'})$ is a characteristic vector starting a bad full-path in $\Gamma'$ then any characteristic vector $K_0 \in \text{Char}(X_{\Gamma})$ restricting to $K_0'$ on $H_2(X_{\Gamma'}; \Z)$ starts a bad full-path on $\Gamma$.
\item Suppose that $\Gamma' \subset \Gamma$ is a connected subgraph whose vertices are all $-2$-weighted. Then for any good full-path $(K_0, \dots , K_n)$ we have that $\langle K, v \rangle = 2$ for at most one vertex $v \in \Gamma'$.
\end{enumerate}
\end{lem}

\begin{lem}\label{famiglia1} For $c\geq 1$ the Brieskorn sphere $\Sigma(2,3,6c-2)$ has one spin structure $\s$ and two conjugated $\Spinc$ structures $\mathfrak{t}$ and $\overline{\mathfrak{t}}$. We have
\[
 d(\Sigma(2,3,6c-2),\mathfrak{s})=\frac{3}{2} \ ,  \text{ and }  \
 d(\Sigma(2,3,6c-2),\mathfrak{t})=d(\Sigma(2,3,6c-2),\overline{\mathfrak{t}})=\frac{1}{6} \ . \]
\end{lem}
\begin{proof}
Let us assume that  $c\geq 2$. We start by describing in full detail the case $c=2$. The Brieskorn sphere $\Sigma(2,3,10)$ can be described via the following negative-definite plumbing graph
 \[ 
\xygraph{
!{<0cm,0cm>;<1cm,0cm>:<0cm,1cm>::}
!~-{@{-}@[|(2.5)]}
!{(-1.2,1.5) }*+{\bullet}="x"
!{(-2.7,1.5) }*+{\bullet}="x1"
!{(-4.2,1.5) }*+{\bullet}="x2"
!{(1.5,3) }*+{\bullet}="a1"
!{(3,3) }*+{\bullet}="am"
!{(1.5,0) }*+{\bullet}="c1"
!{(3,0) }*+{\bullet}="cm"
!{(-1.2,1.9) }*+{-2}
!{(-2.7,1.9) }*+{-2}
!{(-4.2,1.9) }*+{-3}
!{(1.5,3.4) }*+{-2}
!{(1.5,0.4) }*+{-2}
!{(3,3.4) }*+{-2}
!{(3,0.4) }*+{-2}
"x"-"c1"
"x"-"a1"
"a1"-"am"
"c1"-"cm"
"x1"-"x"
"x2"-"x1"
} 
\]

We represent characteristic vectors by recording their value on the vertices of the plumbing graph.
For example the expression
$$
\left[
\begin{array}{lllll}
 &&&0&0\\
 1&0&0&&\\
 &&&0&0\\
\end{array}
\right]
$$
denotes the carachteristic vector which assumes the value 1 at the $-3$-weighted vertex and vanishes on all the other vertices.
According to Lemma \ref{remarks} if a characteristic vector starts a good full-path then its value on the $-3$-weighted vertex
is in $\{\pm 1,3\}$ and it is non-vanishing on at most one more vertex (in this case the corresponding value is necessarily equal to 2).
Therefore, we may list these vectors as follows
\begin{align*}
& \left[
\begin{array}{lllll}
 &&&0&0\\
 \alpha&0&0&&\\
 &&&0&0\\
\end{array}
\right]
,
\left[
\begin{array}{lllll}
 &&&0&0\\
 \alpha&2&0&&\\
 &&&0&0\\
\end{array}
\right]
,
\left[
\begin{array}{lllll}
 &&&0&0\\
 \alpha&0&2&&\\
 &&&0&0\\
\end{array}
\right],
\left[
\begin{array}{lllll}
 &&&2&0\\
 \alpha&0&0&&\\
 &&&0&0\\
\end{array}
\right],\\
& \left[
\begin{array}{lllll}
 &&&0&2\\
 \alpha&0&0&&\\
 &&&0&0\\
\end{array}
\right]
,
\left[
\begin{array}{lllll}
 &&&0&0\\
 \alpha&0&0&&\\
 &&&2&0\\
\end{array}
\right],
\left[
\begin{array}{lllll}
 &&&0&0\\
 \alpha&0&0&&\\
 &&&0&2\\
\end{array}
\right],
\end{align*}
where $\alpha\in\{\pm 1,3\}$. The only characteristic vectors which start good full-paths are
$$
\left[
\begin{array}{lllll}
 &&&0&0\\
 \pm1&0&0&&\\
 &&&0&0\\
\end{array}
\right],
\left[
\begin{array}{lllll}
 &&&0&2\\
 -1&0&0&&\\
 &&&0&0\\
\end{array}
\right],
\left[
\begin{array}{lllll}
 &&&0&0\\
 -1&0&0&&\\
 &&&0&2\\
\end{array}
\right].
$$
The first two vectors belong to a good full-path of length one. The third vector belongs to the following good full-path:
\begin{align*}
& \left[
\begin{array}{lllll}
 &&&0&\textbf{2}\\
 -1&0&0&&\\
 &&&0&0\\
\end{array}
\right]
\sim
\left[
\begin{array}{rrrrr}
 &&&\textbf{2}&-2\\
 -1&0&0&&\\
 &&&0&0\\
\end{array}
\right] 
\sim
\left[
\begin{array}{rrrrr}
 &&&-2&0\\
 -1&0&\textbf{2}&&\\
 &&&0&0\\
\end{array}
\right] \sim \\
&\left[
\begin{array}{rrrrr}
 &&&0&0\\
 -1&\textbf{2}&-2&&\\
 &&&2&0\\
\end{array}
\right] \sim 
\left[
\begin{array}{rrrrr}
 &&&0&0\\
 1&-2&0&&\\
 &&&\textbf{2}&0\\
\end{array}
\right] \sim
\left[
\begin{array}{rrrrr}
&&&0&0\\
1&-2&\textbf{2}&&\\
&&&-2&2\\
\end{array}
\right] \sim \\
&\left[
\begin{array}{rrrrr}
 &&&2&0\\
1&0&-2&&\\
&&&0&\textbf{2}\\
\end{array}
\right] \sim 
\left[
\begin{array}{rrrrr}
 &&&\textbf{2}&0\\
1&0&-2&&\\
&&&2&-2\\
\end{array}
\right] \sim 
\left[
\begin{array}{rrrrr}
 &&&-2&2\\
1&0&0&&\\
&&&\textbf{2}&-2\\
\end{array}
\right] \sim \\
 &\left[
\begin{array}{rrrrr}
 &&&-2&2\\
1&0&\textbf{2}&&\\
&&&-2&0\\
\end{array}
\right] \sim 
\left[
\begin{array}{rrrrr}
 &&&0&2\\
1&\textbf{2}& -2&&\\
&&&0&0\\
\end{array}
\right] \sim 
\left[
\begin{array}{rrrrr}
 &&&0& \textbf{2}\\
3&-2& 0&&\\
&&&0&0\\
\end{array}
\right] \sim \\
&\left[
\begin{array}{rrrrr}
 &&&\textbf{2}& -2\\
3&-2& 0&&\\
&&&0&0\\
\end{array}
\right] \sim
\left[
\begin{array}{rrrrr}
 &&&-2& 0\\
3&-2& \textbf{2}&&\\
&&&0&0\\
\end{array}
\right] \sim
\left[
\begin{array}{rrrrr}
 &&&0& 0\\
3&0& -2&&\\
&&&\textbf{2}&0\\
\end{array}
\right] \sim \\
&\left[
\begin{array}{rrrrr}
 &&&0& 0\\
3&0& 0&&\\
&&&-2&\textbf{2}\\
\end{array}
\right] \sim 
\left[
\begin{array}{rrrrr}
 &&&0& 0\\
\textbf{3}&0& 0&&\\
&&&0&-2\\
\end{array}
\right]  \sim
\left[
\begin{array}{rrrrr}
 &&&0& 0\\
-3&\textbf{2}& 0&&\\
&&&0&-2\\
\end{array}
\right] \sim \\
&\left[
\begin{array}{rrrrr}
 &&&0& 0\\
-1&-2& \textbf{2}&&\\
&&&0&-2\\
\end{array}
\right] \sim
\left[
\begin{array}{rrrrr}
 &&&\textbf{2}& 0\\
-1&0& -2&&\\
&&&2&-2\\
\end{array}
\right] \sim 
\left[
\begin{array}{rrrrr}
 &&&-2& 2\\
-1&0& 0&&\\
&&&\textbf{2}&-2\\
\end{array}
\right] \sim \\
&\left[
\begin{array}{rrrrr}
 &&&-2& 2\\
-1&0& \textbf{2}&&\\
&&&-2&0\\
\end{array}
\right] \sim
\left[
\begin{array}{rrrrr}
 &&&0& 2\\
-1&\textbf{2}&-2&&\\
&&&0&0\\
\end{array}
\right]  \sim 
\left[
\begin{array}{rrrrr}
 &&&0&\textbf{2}\\
1&-2&0&&\\
&&&0&0\\
\end{array}
\right] \sim \\
&\left[
\begin{array}{rrrrr}
 &&&\textbf{2}&-2\\
1&-2&0&&\\
&&&0&0\\
\end{array}
\right] \sim 
\left[
\begin{array}{rrrrr}
 &&&-2&0\\
1&-2&\textbf{2}&&\\
&&&0&0\\
\end{array}
\right] \sim
\left[
\begin{array}{rrrrr}
 &&&0&0\\
1&0&-2&&\\
&&&\textbf{2}&0\\
\end{array}
\right] \sim \\
&\left[
\begin{array}{rrrrr}
 &&&0&0\\
1&0&0&&\\
&&&-2&\textbf{2}\\
\end{array}
\right] \sim  
\left[
\begin{array}{rrrrr}
 &&&0&0\\
1&0&0&&\\
&&&0&-2\\
\end{array}
\right] .
\end{align*}
In the full-path above the boldfaced coefficients are the ones associated with the vertices on which the flip move is performed. Because of the obvious symmetry of the plumbing graph we also have a good full-path
$$\left[
\begin{array}{lllll}
 &&&0&0\\
 -1&0&0&&\\
 &&&0&2\\
\end{array}
\right]
\sim 
\dots \sim
\left[
\begin{array}{rrrrr}
 &&&0&-2\\
1&0&0&&\\
&&&0&0\\
\end{array}
\right] . $$

Via a straightforward but tedious computation one can find bad full-paths for all the other possible charcteristic vectors. Here we illustrate a specific example which will be useful later on in this proof.
\begin{align*}
& \left[
\begin{array}{rrrrr}
 &&&0&0\\
\textbf{3}&0&0&&\\
&&&0&0\\
\end{array}
\right] \sim
 \left[
\begin{array}{rrrrr}
 &&&0&0\\
-3&\textbf{2}&0&&\\
&&&0&0\\
\end{array}
\right] \sim
\left[
\begin{array}{rrrrr}
 &&&0&0\\
-1&-2&\textbf{2}&&\\
&&&0&0\\
\end{array}
\right] \sim \\
& \left[
\begin{array}{rrrrr}
 &&&\textbf{2}&0\\
-1&0&-2&&\\
&&&\textbf{2}&0\\
\end{array}
\right] \sim  
\left[
\begin{array}{rrrrr}
 &&&-2&2\\
-1&0&\textbf{2}&&\\
&&&-2&2\\
\end{array}
\right] \sim 
\left[
\begin{array}{rrrrr}
 &&&0&\textbf{2}\\
-1&2&-2&&\\
&&&0&\textbf{2}\\
\end{array}
\right] \sim \\
& \left[
\begin{array}{rrrrr}
 &&&\textbf{2}&-2\\
-1&2&-2&&\\
&&&\textbf{2}&-2\\
\end{array}
\right] \sim 
\left[
\begin{array}{rrrrr}
 &&&-2&0\\
-1&2&\textbf{2}&&\\
&&&-2&0\\
\end{array}
\right] \sim 
\left[
\begin{array}{rrrrr}
 &&&0&0\\
-1&4&-2&&\\
&&&0&0\\
\end{array}
\right] .
\end{align*}

When $c\geq 3$ the Brieskorn sphere $\Sigma(2,3,6c-2)$ can be described via the following negative-definite plumbing graph
 \[ 
\xygraph{
!{<0cm,0cm>;<1cm,0cm>:<0cm,1cm>::}
!~-{@{-}@[|(2.5)]}
!{(-1.2,1.5) }*+{\bullet}="x"
!{(-2.7,1.5) }*+{\bullet}="x1"
!{(-4.2,1.5) }*+{\bullet}="x2"
!{(-5.7,1.5) }*+{\bullet}="x3"
!{(-7.2,1.5) }*+{\dots}="x4"
!{(-8.7,1.5) }*+{\bullet}="x5"
!{(1.5,3) }*+{\bullet}="a1"
!{(3,3) }*+{\bullet}="am"
!{(1.5,0) }*+{\bullet}="c1"
!{(3,0) }*+{\bullet}="cm"
!{(-1.2,1.9) }*+{-2}
!{(-2.7,1.9) }*+{-2}
!{(-4.2,1.9) }*+{-3}
!{(-5.7,1.9) }*+{-2}
!{(-8.7,1.9) }*+{-2}
!{(1.5,3.4) }*+{-2}
!{(1.5,0.4) }*+{-2}
!{(3,3.4) }*+{-2}
!{(3,0.4) }*+{-2}
"x"-"c1"
"x"-"a1"
"a1"-"am"
"c1"-"cm"
"x1"-"x"
"x2"-"x1"
"x3"-"x2"
"x4"-"x3"
"x5"-"x4"
}
\]
where the length of the leftmost $-2$-chain is $c-2$. It follows from our previous argument (when $c=2$) and from the second statement in Lemma \ref{remarks} that if a characteristic vector starts a good full-path then it belongs to the following list
$$
\left[
\begin{array}{rrrrrrrr}
 &&&&&&0&0\\
x_1&\dots&x_{c-2}&\pm 1&0&0&&\\
&&&&&&0&0\\
\end{array}
\right],
\left[
\begin{array}{rrrrrrrr}
 &&&&&&0&2\\
x_1&\dots&x_{c-2}& -1&0&0&&\\
&&&&&&0&0\\
\end{array}
\right],
$$

$$
\left[
\begin{array}{rrrrrrrr}
 &&&&&&0&0\\
x_1&\dots&x_{c-2}&-1&0&0&&\\
&&&&&&0&2\\
\end{array}
\right],
$$
where at most one of the $x_i$'s is non-zero (and if so it is necessarily equal to $2$).

A characteristic vector of the form
$$K_0=
\left[
\begin{array}{rrrrrrrr}
 &&&&&&0&0\\
x_1&\dots&x_{c-2}&1&0&0&&\\
&&&&&&0&0\\
\end{array}
\right]
$$
does not start a good full-path if one of the $x_i$'s is non-zero. To see this suppose that we have $x_i=2$ for some 
$i$. Then, we can write down the following bad full-path
\begin{small}
\begin{align*}
K_0&\sim\dots\sim
\left[
\begin{array}{rrrrrrrr}
 &&&&&&0&0\\
\dots&-2&\textbf{2}&1&0&0&&\\
&&&&&&0&0\\
\end{array}
\right]\sim
\left[
\begin{array}{rrrrrrrr}
 &&&&&&0&0\\
\dots&0&-2&\textbf{3}&0&0&&\\
&&&&&&0&0\\
\end{array}
\right]\sim \\
& \hspace{0.957cm} \sim \left[
\begin{array}{rrrrrrrr}
 &&&&&&0&0\\
\dots &0&0&-3&\textbf{2}&0&&\\
&&&&&&0&0\\
\end{array}
\right]\sim\dots
\sim\left[
\begin{array}{rrrrrrrr}
 &&&&&&0&0\\
\dots&0&0&-1&4&-2&&\\
&&&&&&0&0\\
\end{array}
\right] \ .
\end{align*}
\end{small}
\hspace{-0.3cm} where the first omitted sequence of moves is the obvious sequence of flips which starts with a flip move 
at $x_i=2$, while the second one is the one suggested by the full-path 
$$\left[
\begin{array}{rrrrr}
 &&&0&0\\
3&0&0&&\\
&&&0&0\\
\end{array}
\right]  \sim
\dots \sim 
\left[
\begin{array}{rrrrr}
 &&&0&0\\
-1&4&-2&&\\
&&&0&0\\
\end{array}
\right]$$
described above.

Similarly, the characteristic vector
$$\left[
\begin{array}{rrrrrrrr}
 &&&&&&0&2\\
x_1&\dots&x_{c-2}& -1&0&0&&\\
&&&&&&0&0\\
\end{array}
\right]$$
does not start a good full-path if one of the $x_i$'s is non-zero. In fact, if $x_i=2$ for some $i$ then the (good) 
full-path 
$$\left[
\begin{array}{rrrrr}
 &&&0&2\\
-1&0&0&&\\
&&&0&0\\
\end{array}
\right]  \sim
\dots \sim 
\left[
\begin{array}{rrrrr}
 &&&0&0\\
1&0&0&&\\
&&&0&-2\\
\end{array}
\right]$$
induces a sequence of flip moves  
\begin{small}
\begin{align*}
\left[
\begin{array}{rrrrrrrr}
 &&&&&&0&2\\
x_1&\dots&x_{c-2}& -1&0&0&&\\
&&&&&&0&0\\
\end{array}
\right] \sim \dots
\sim \left[
\begin{array}{rrrrrrrr}
 &&&&&&0&0\\
x_1&\dots&x_{c-2}+2& 1&0&0&&\\
&&&&&&0&-2\\
\end{array}
\right]
\end{align*}
\end{small}
leading to a bad full-path 
\begin{small}
\begin{align*}
\left[
\begin{array}{rrrrrrrr}
 &&&&&&0&2\\
x_1&\dots&x_{c-2}& -1&0&0&&\\
&&&&&&0&0\\
\end{array}
\right] \sim \dots \sim 
\left[
\begin{array}{rrrrrrrr}
&&&&&&0&0\\
\dots&4&\dots& 3&0&0&&\\
&&&&&&0&-2\\
\end{array}
\right] .
\end{align*}
\end{small}
Symmetrically one may find a bad full-path
\begin{small}
\begin{align*}
\left[
\begin{array}{rrrrrrrr}
 &&&&&&0&0\\
x_1&\dots&x_{c-2}& -1&0&0&&\\
&&&&&&0&2\\
\end{array}
\right] \sim \dots \sim 
\left[
\begin{array}{rrrrrrrr}
&&&&&&0&-2\\
\dots&4&\dots& 3&0&0&&\\
&&&&&&0&0\\
\end{array}
\right] .
\end{align*}
\end{small} 
when at least one of the $x_i$'s is equal to $2$. 

Summarising we found that only a characteristic vector of the form 
$$
\left[
\begin{array}{rrrrrrrr}
 &&&&&&0&0\\
0&\dots&0&-1&0&0&&\\
&&&&&&0&2\\
\end{array}
\right],
\left[
\begin{array}{rrrrrrrr}
 &&&&&&0&2\\
0&\dots&0& -1&0&0&&\\
&&&&&&0&0\\
\end{array}
\right],
$$
$$
\left[
\begin{array}{rrrrrrrr}
 &&&&&&0&0\\
0&\dots&0&1&0&0&&\\
&&&&&&0&0\\
\end{array}
\right],
\left[
\begin{array}{rrrrrrrr}
 &&&&&&0&0\\
x_1&\dots&x_{c-1}&1&0&0&&\\
&&&&&&0&0\\
\end{array}
\right],
$$
with at most one of the $x_i$'s is non-zero and equal to $2$, can start a good full path. In fact one can easily check that all of them do. By means of Equation \ref{correctionterms} one can now conclude the argument (we omit the easy but tedious computations).   
\end{proof}

\bibliography{Bibliotesi}
\bibliographystyle{siam}

\end{document}